\documentclass[12pt,reqno]{article}

\usepackage[usenames]{color}
\usepackage{amssymb}
\usepackage{bbm}
\usepackage{graphicx}
\usepackage{amscd}

\usepackage[colorlinks=true,
linkcolor=webgreen,
filecolor=webbrown,
citecolor=webgreen]{hyperref}

\definecolor{webgreen}{rgb}{0,.5,0}
\definecolor{webbrown}{rgb}{.6,0,0}

\usepackage{color}
\usepackage{fullpage}
\usepackage{float}
\usepackage{caption}

\usepackage{graphics,amsmath,amssymb}
\usepackage{amsthm}
\usepackage{amsfonts}
\usepackage{latexsym}
\usepackage{epsf}

\makeatletter

\theoremstyle{plain}
\newtheorem{theorem}{Theorem}

\newtheorem{corollary}[theorem]{Corollary}
\newtheorem{remark}[theorem]{Remark}

\theoremstyle{definition}
\newtheorem{definition}[theorem]{Definition}

\newtheorem{Theorem}{Theorem}[section]
\theoremstyle{definition}

\newtheorem{prop}[Theorem]{Proposition}

\makeatother

\begin{document}

\begin{center}
\vskip 1cm{\LARGE\bf
A bijection between triangulations and $312$-avoiding permutations}\\
\large
Alon Regev\\
Department of Mathematical Sciences\\
Northern Illinois University\\
DeKalb, IL\\
\end{center}


\section{Introduction}

A permutation $a_1a_2\ldots a_n$ is {\em $312$-avoiding} if for $i<j<k$ we never have $a_j<a_k<a_i$. It is well known that the sets $S_n(312)$ of total number of $312$-avoiding permutations of $\{1,\ldots,n\}$ are eumerated by the Catalan numbers $\displaystyle C_n={1\over n+1}{2n\choose n}$. Equivalently, this is the number of permutations in $S_n(132)$, the $132$-avoiding permutations in $S_n$. Several bijective proofs that $|S_n(132)|=C_n$ are presented in the paper of Claesson and Kitaev \cite{CK}, which provides a survey and analysis of bijections between $132$-avoiding permutations and $123$-avoiding permutations.

In this paper we give a bijection between the set of triangulations of an $n$-gon and the $312$-avoiding permutations in $S_{n-2}$, using what we call the {\em clip sequence} of a triangulation. We also use this concept to find a bijection between all dissections of a polygon and certain classes of $312$-avoiding permutations, including alternating $312$-avoiding permutations and their generalizations.

\section{Clip sequences}

 A triangulation of a convex $n$-gon may be considered as a graph with $n$ vertices labeled by positive integers.
The
{\em 
clip sequence} of a triangulation is defined as the output of the following procedure: at each step, delete (and record) the vertex of degree $2$ having the smallest label, and its two incident edges (see Figures \ref{1234} - \ref{132564}). It is easy to see that the result is a permutation  of $\{1,\ldots,n-2\}$. 

\begin{theorem}\label{312clip}
The clip sequences of triangulations of $n$-gons are precisely the $312$-avoiding permutations of $\{1,2,\ldots, n-2\}$. The map taking every triangulation to its clip-sequence is a bijection between the triangulations of an $n$-gon and $S_{n-2}(312)$.
\end{theorem}
The proof uses the following property of $312$-avoiding permutations.
\begin{prop}\label{property}
 A permutation of $\{1,\ldots,m\}$ is $312$-avoiding if and only if it has the form $a_1\ldots a_{k-1}b_{k+1}\ldots b_{m-1}k$, where $\{a_i\}$ and $\{b_i\}$ are (possibly empty) $312$-avoiding permutations of $\{1,\ldots, k-1\}$ and of $\{k+1,\ldots, m-1\}$, respectively.
\end{prop}

\begin{proof}[{\bf Proof of Theorem \ref{312clip}}:]
The proof is by induction on $n$, with the case $n=2$ corresponding to the empty clip sequence. Given a clip sequence $a_1a_2\ldots a_{n-2}$ of a triangulation, let $k=a_{n-2}$, and note that $k$ is the triangle with vertices $k,n-1,n$ is one of the triangles of the triangulation. This defines a triangulation of the polygon with vertices $1,\ldots, k-1,k,n$ and a triangulation of the polygon with vertices $k,k+1,\ldots, n-1,n$. We claim that $a_1\ldots a_{k-1}$ is a $312$-avoiding permutation of $\{1,\ldots, k-1\}$. Indeed, since every triangulation has at least two ears, at each of the first $k-1$ iterations of the clipping process, there will be an ear numbered $a_i\in \{1,\ldots,k-1\}$ in the triangulation of the $n$-gon. Since those ears have a smaller label than any of the other ears of the triangulation of the $n$-gon, they are the ones that will be clipped, until they are all removed. Once these ears are removed, the triangulation remaining has vertices $k,k+1,\ldots, n+2$, and so the next ears clipped are the ones numbered $b_i\in \{k+1,\ldots, n-2\}$. By induction, both sequences $\{a_i\}$ and $\{b_i\}$ are $312$-avoiding, and so by Proposition \ref{property} the result is a $312$-avoiding permutation.

Conversely, by Proposition \ref{property} each $312$-avoiding permutation of $n-2$ has the form $a_1\ldots a_{k-1}b_{k+1}\ldots b_{n-3}k$, as above. By induction, the sequences $\{a_i\}$ and $\{b_i\}$ correspond to clip sequences of triangulations. Attaching these triangulations on each side of the triangle with vertices $n-1,n,k$ then gives a triangulation whose clip sequence is the given permutation.
\end{proof}

\begin{figure}[h]
\begin{center}
\epsfxsize=6.7in
\leavevmode\epsffile{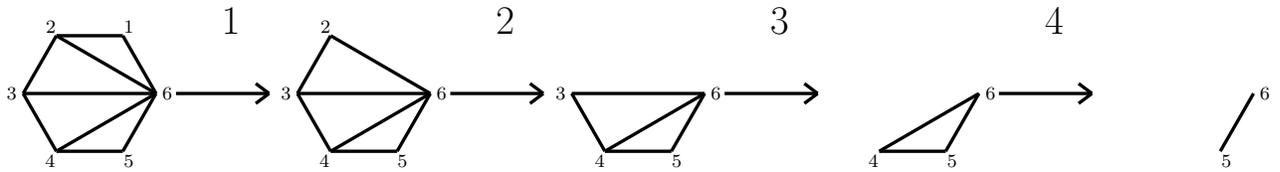}
\end{center}
\caption{Clip sequence corresponding to the permutation $1234$}\label{1234}
\end{figure}

\begin{figure}[h]
\begin{center}
\epsfxsize=6.7in
\leavevmode\epsffile{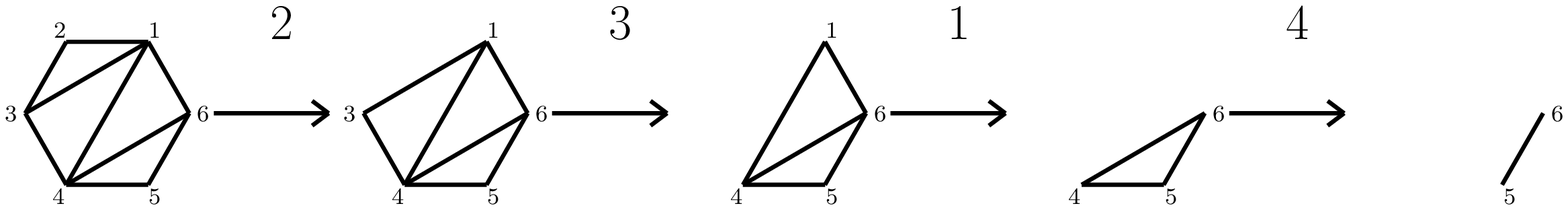}
\end{center}
\caption{Clip sequence corresponding to the permutation $2314$}\label{2314}
\end{figure}


\begin{figure}[H]
\begin{center}
\epsfxsize=6.7in
\leavevmode\epsffile{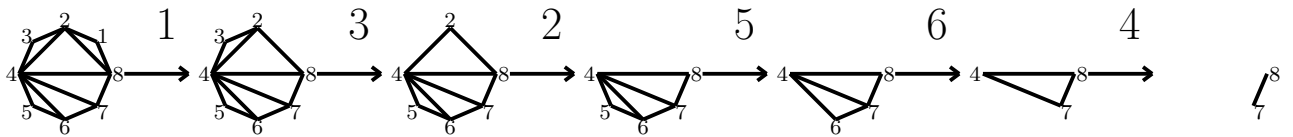}
\end{center}
\caption{Clip sequence corresponding to the permutation $132564$}\label{132564}
\end{figure}

\begin{remark}
The bijection given here can also be described as follows. Each triangulation of an $n$-gon can be associated with a binary tree via a standard bijection (see, e.g. \cite[Corollary 6.2.3]{S}), and the reference edge chosen to be $(n-1)n$ . Reading the entries of the resulting binary tree in post-order then gives the clip sequence of the triangulation.
\end{remark}

\section{Polygon dissections and decent permutations}

A {\em dissection} of an $n$-gon is a partition of the polygon by at most $n-3$ diagonals. In this section we give a bijection between the dissections of a polygon and a class of $312$-avoiding permutations.
\begin{definition}
The {\em up/down pattern} of a permutation $a_1\ldots a_n$ is the word $x_1x_2\ldots x_{n-1}$, where $x_i=D$ if $a_i>a_{i+1}$ and $x_i=U$ if $a_i<a_{i+1}$.
A permutation is {\em decent} if its up/down pattern has the form $D^{i_1}U D^{i_2}U \cdots U D^{i_k}$, with all $i_t\ge 1$.
\end{definition}

Equivalently, a permutation is decent if its up/down pattern starts and ends with $D$  and avoids any occurrence of $UU$.

Consider the following procedure. Let $\pi=a_1\ldots a_n$ be a decent $312$-avoiding permutation. Denote
\begin{align*}
a_1 \ldots a_n= c_{11}c_{12}\ldots c_{1i_1}c_{21}  c_{22} \ldots  c_{2i_2} \ldots \ldots  c_{k1}c_{k2} \ldots c_n,
\end{align*}
where
\begin{align*}
c_{11}> c_{12}>\ldots >c_{1i_1}< c_{21} >c_{22}>\ldots >c_{2i_2} <\ldots \ldots <  c_{k1} >c_{k2}>\ldots > c_n.
\end{align*}
By Theorem \ref{312clip},  $\pi$ is a clip sequence of a triangulation. Consider the triangles which, in the clip sequence, were ears numbered $c_{12},c_{13}, \ldots , c_{1i_1}$. Since their numbers are decreasing, these $i_1$ triangles form an $(i_1+2)$-gon.
Similarly, each set of triangles $c_{r2},\ldots,c_{ri_r}$ forms an $(i_r+2)$-gon.
Each of the remaining triangles corresponding to ears numbered $c_{11},c_{21},\ldots, c_{k1}$ is either an ear or shares one side with the polygon (it cannot be an inner triangle since that would imply $c_{i1}$ is immediately preceded by a smaller number). Those triangles which share a side with the polygons, create ``gaps" in the polygon. These gaps are then ``closed", by identifying every two vertices that have become disconnected.
Finally, changing the vertex labels to $1,\ldots,n-k$ results in a dissection of the $(n-k)$-gon with vertices $1,\ldots , n-k$ (see Figures \ref{CS435261}-\ref{CS654873921}).

The procedure described above describes a map from the set of decent $312$-avoiding permutations of $n$ to the set of polygon dissections. This map is in fact a bijection, since each of the steps of the procedure is reversible. Thus we have the following theorem.


\begin{theorem}\label{Main}
The decent $312$-avoiding permutations of $\{1,\ldots, n\}$ having up/down pattern $D^{i_1}U D^{i_2}U \cdots U D^{i_k}$
are in bijection with the dissections of an $(n-k+1)$-gon by $k$ diagonals.  
\end{theorem}

The following corollaries follow easily from Theorem \ref{Main}. 
We first obtain a result of Lewis \cite{L}. Note that in this case the resulting bijection is essentially equivalent to the bijection given by Lewis, and may be considered as a visualization of his bijection.

\begin{corollary}\cite{L}
The alternating, $312$-avoiding permutations of $\{1,\ldots, 2n\}$ (i.e., those with up-down pattern $DUDU\ldots DUD$) are in bijection with the triangulations of an $(n+2)$-gon, and hence with the $312$-avoiding permutations of $\{1,\ldots, n\}$.
\end{corollary}

\begin{corollary}
The $312$-avoiding permutations 
of $\{1,\ldots, (j+1)m\}$ 
with up/down pattern  $D^jUD^jU\ldots UD^j$ are in bijection with the $(j+2)$-angulations of an $(mj+2)$-gon.
\end{corollary}

\begin{figure}[H]
\begin{center}
\epsfxsize=7in
\leavevmode\epsffile{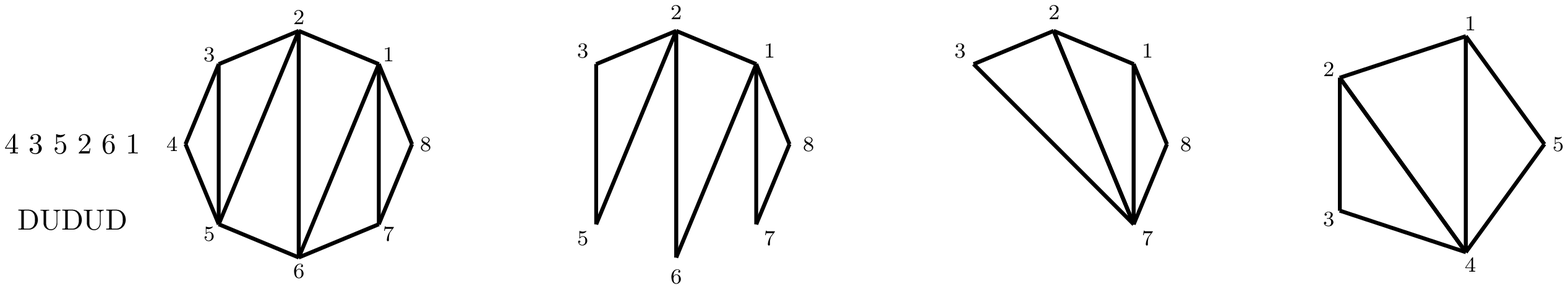}
\end{center}
\caption{}\label{CS435261}
\end{figure}

\begin{figure}[H]
\begin{center}
\epsfxsize=7in
\leavevmode\epsffile{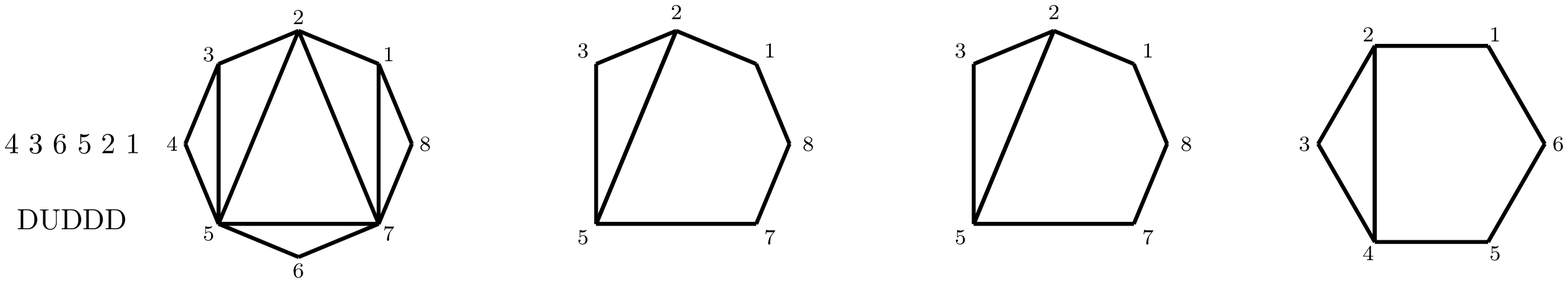}
\end{center}
\caption{}\label{CS436521}
\end{figure}

\begin{figure}[H]
\begin{center}
\epsfxsize=7in
\leavevmode\epsffile{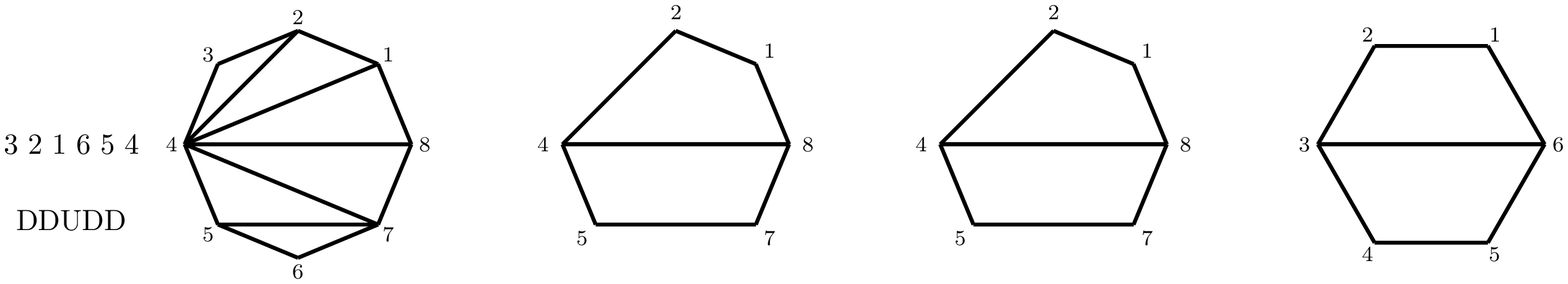}
\end{center}
\caption{}\label{CS543621}
\end{figure}

\begin{figure}[H]
\begin{center}
\epsfxsize=7in
\leavevmode\epsffile{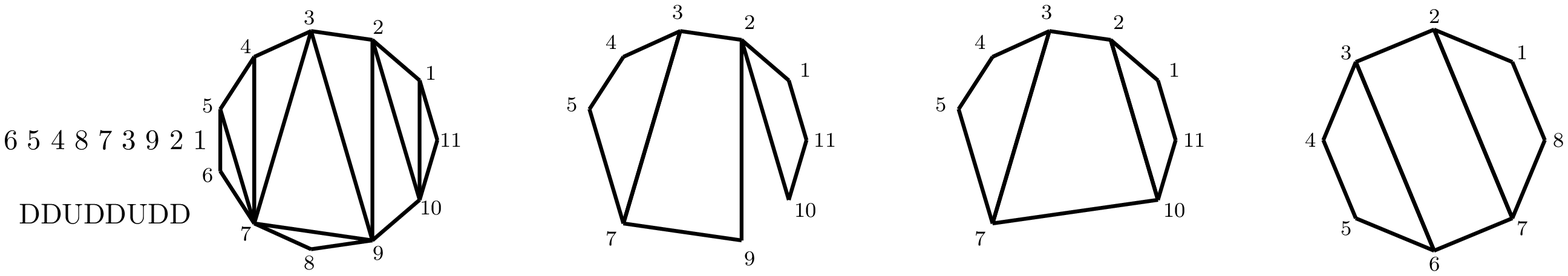}
\end{center}
\caption{}\label{CS654873921}
\end{figure}


\section{Acknowledgement}
The author thanks Dennis Stanton and Dennis White, who pointed out the alternative interpretation of the bijection of Theorem \ref{312clip}.


\begin{thebibliography}{99}

\bibitem{CK} A. Claesson and S. Kitaev, Classification of bijections between
321- and 132-avoiding permutations,  S\'eminaire Lotharingien de Combinatoire 60: B60d.


\bibitem{L} J. B. Lewis, Alternating, pattern-avoiding permutations, Electron. J. Combin. {\bf 16} (2009).








\bibitem{S} R. P. Stanley, Enumerative Combinatorics, Vol. 2, 2nd edition, Cambridge University Press (1997).



\end{thebibliography}
\end{document}